\documentclass[11pt,reqno]{amsart}

\usepackage[utf8]{inputenc}
\usepackage[T1]{fontenc}
\usepackage{amsmath,amssymb,amsthm}
\usepackage{mathtools}
\usepackage[margin=1.05in]{geometry}
\usepackage{booktabs}
\usepackage{tabularx}
\usepackage{microtype}
\microtypesetup{expansion=false}
\usepackage{enumitem}
\usepackage{tikz}
\usepackage{xcolor}
\usepackage[colorlinks=true,linkcolor=blue!45!black,citecolor=blue!45!black,
  urlcolor=blue!45!black]{hyperref}

\theoremstyle{plain}
\newtheorem{theorem}{Theorem}[section]
\newtheorem{lemma}[theorem]{Lemma}
\newtheorem{proposition}[theorem]{Proposition}
\newtheorem{corollary}[theorem]{Corollary}

\theoremstyle{definition}
\newtheorem{definition}[theorem]{Definition}

\newtheorem{example}[theorem]{Example}

\theoremstyle{remark}
\newtheorem{remark}[theorem]{Remark}

\newcommand{\N}{\mathbb N}

\newcommand{\Z}{\mathbb Z}
\newcommand{\R}{\mathbb R}

\newcommand{\Qcube}{Q_3}
\newcommand{\W}{\mathcal W}
\newcommand{\A}{\mathcal A}
\newcommand{\U}{\mathcal U}
\newcommand{\occ}{\operatorname{occ}}

\newcommand{\ab}{\operatorname{ab}}

\newcommand{\gen}[1]{x_{#1}}
\newcolumntype{Y}{>{\raggedright\arraybackslash}X}
\begin{document}
\raggedbottom

\title[From One Generator to Loop Order on the Three-Cube]{From One Generator to Loop Order on the Three-Dimensional Cube}

\author{Jonathan Washburn}
\address{Recognition Physics Institute, Austin, Texas, USA}
\email{jon@recognitionphysics.org}

\author{Milan Zlatanovi\'c}
\address{Department of Mathematics, Faculty of Science and Mathematics,
University of Ni\v{s}, Vi\v{s}egradska 33, 18000 Ni\v{s}, Serbia}
\email{zlatmilan@yahoo.com}

\date{}

\begin{abstract}
We study the monoid \(\W\) of based closed walks on the \(1\)-skeleton
of the three-dimensional cube \(Q_3\). The fundamental group of this
graph is the free group \(F_5\). We prove that every additive one-step reading, given by a sum of edge weights in an abelian group, factors through the
directed transition counts, and we exhibit two closed walks with equal
transition counts and different reduced loop words. Hence commutative
aggregation does not determine the reduced loop word.

The reduced loop word gives a proper noncommutative recognition
quotient. We also determine the shortest closed walk with trivial
abelianization but nontrivial degree-two commutator information. Its
minimum length in the cube edge metric is \(14\).  

We also obtain a separation between finite and unbounded memory. A
two-state reading separates an order pair, while for every \(k\geq1\),
there is an explicit factorial pair which no \(k\)-state reading separates.
An unbounded stack recovers the reduced loop word on every walk.

For integer-valued functions on the vertex set, potential readings vanish on closed walks, while occupation binding is not
rectangular. More generally, every occupation-based constraint is
rectangular on a class of histories if and only if the occupation
vector is constant on that class.

Finally, a declared quarter-turn quaternion clock gives a second proper
order-sensitive congruence, incomparable with the reduced-word
quotient. The obstruction to commutative recovery, and the minimum
\(14\), remain valid on every hypercube \(Q_n\), \(n\geq3\).
\\

\smallskip

\noindent{{\bf AMS class. 2020:} 20F14, 20F34, 20M35}
\end{abstract}

\maketitle

\section{Introduction}\label{sec:intro}

In~\cite{delta-paper} we studied the structure generated from an empty record \(0\) by a single repeatable step \(t\mapsto St\), and showed that the orbit generated is a free additive monoid on one generator, isomorphic to \((\N,+,0)\). A recognizer of a monoid \(M\) is a monoid
homomorphism
$$
r:M\longrightarrow N,
$$
where the monoid \(N\) need not be finite. This is more general than
the finite monoid convention used in~\cite{eilenberg}. The relation
$$
x\sim_r y
\quad\Longleftrightarrow\quad
r(x)=r(y)
$$
is the kernel congruence of \(r\), and \(M/{\sim_r}\) is its recognition
quotient.

Under excluded middle, in \cite{delta-paper}, we classified the recognition
quotients of this monogenic monoid: every recognizer is either injective or
has a finite index-period quotient \(M(i,p)\). Since the carrier is generated
by \(S0\), every monoid recognizer is determined by the image of this one
generator, so there is no independent order information. We now consider a
carrier (the monoid whose elements are to be recognized)  with several independent loops, where order can matter.

We take as carrier the monoid \(\W\) of based closed walks (we also call them histories) on the
\(1\)-skeleton of the three-dimensional cube \(Q_3\). The restriction to
based closed walks is a convention, not a consequence of
recognition. We fix a spanning tree 
and
obtain a homomorphism
$$
\rho:\W\longrightarrow F_5,
$$ where \(F_5\) is the free group of rank five, 
which assigns to each closed walk its reduced loop word. Thus \(Q_3\) is
the first hypercube whose reduced loop words can be noncommutative
(Proposition~\ref{prop:firstcube}). We also use \(Q_3\) because its
edges are divided into three coordinate directions. The quarter-turn clock of Theorem~\ref{thm:quarter-clock}
uses exactly these three directions. The tetrahedron \(K_4\) has four vertices and six edges, hence cycle
rank three and \(\pi_1(K_4)\cong F_3\). The graph \(K_4\) is smaller and has
a nonabelian fundamental group, but it does not have this three-coordinate
product structure.
\begin{definition}\label{def:clock}
A sequential reading consists of a state set \(S\), an initial state
\(s_0\in S\), update maps \(T_e:S\to S\) for directed edges \(e\), a set
\(Y\), and an output map \(q:S\to Y\). On a walk whose directed edge sequence is \(e_1,\ldots,e_n\), its value is
$$
q\bigl(T_{e_n}\circ\cdots\circ T_{e_1}(s_0)\bigr).
$$
We call such a sequential reading a \emph{clock}.  The state set may be finite or
infinite. For an additive reading, we take \(S\) to be an abelian group and
$$
T_e(s)=s+g(e).
$$
\end{definition}

A sequential reading need not be a monoid recognizer. Here, we use the term
{\it recognition quotient} only for monoid recognizers.

Let
$$
\ab:F_5\longrightarrow \Z^5
$$
be the abelianization map, and put
\begin{equation}\label{11}
a=\ab\circ\rho.
\end{equation}

We construct two closed walks with the same directed transition counts.
Every additive one-step reading takes the same value on these walks,
while their reduced loop words are different. 
So, directed transition counts, and more generally any commutative
aggregation, do not determine the reduced loop word (Section~\ref{sec:blindness}).

We next compare finite-state readings with unbounded memory. Let
\(
B=(000;1212).
\) For every
\(k\geq1\) we construct two closed walks
$$
B^{k!},\qquad B^{2\cdot k!},
$$
such that every \(k\)-state reading \(F\) (every sequential reading with at most \(k\) states) satisfies 
\[
F\bigl(B^{k!}\bigr)=F\bigl(B^{2\cdot k!}\bigr),
\]
while
\[
\rho\bigl(B^{k!}\bigr)\neq\rho\bigl(B^{2\cdot k!}\bigr).
\]
This pair is already separated by the abelian reading \(a\). What fails
is any fixed finite-state bound (Section~\ref{sec:hierarchy}). An unbounded stack, which stores the current reduced word and cancels adjacent inverse pairs, performs free reduction and recovers
\(\rho(w)\) for every \(w\in\W\).

The induced congruences satisfy
$$
\Delta_{\W}
\subsetneq
\sim_{\rho}
\subsetneq
\sim_{a},
$$
where \(\Delta_{\W}\) is equality on \(\W\). Thus \(\rho\) is not
injective, but its kernel is strictly finer than that of \(a\). In
particular, \(\rho\) defines a proper noncommutative recognition
quotient of \(\W\) (Section~\ref{sec:congruence}).

We denote by \(|w|\)  the number of steps of a walk \(w\). The shortest based closed walk with trivial abelianization but
nontrivial degree-two commutator information has length \(14\)
(Section~\ref{sec:fourteen}). More precisely,
$$
\min\bigl\{
|w|:\rho(w)\in\gamma_2(F_5)\setminus\gamma_3(F_5)
\bigr\}=14.
$$
Here \(\gamma_1(F_5)=F_5\) and
\(\gamma_{n+1}(F_5)=[\gamma_n(F_5),F_5]\), so
\(\gamma_2(F_5)=[F_5,F_5]\). Thus \(\rho(w)\in\gamma_2(F_5)\) means that its abelianization is trivial, while
\(\rho(w)\notin\gamma_3(F_5)\) means that its class in
\(\gamma_2(F_5)/\gamma_3(F_5)\) is nontrivial.

A declared quarter-turn transport, defined by assigning a quaternion
factor to each of the three coordinate directions, gives a second proper
order-sensitive congruence (see Theorem~\ref{thm:quarter-clock}). This
congruence is incomparable with \(\sim_{\rho}\) (see
Proposition~\ref{prop:clock-incomparable}).

\section{The three-dimensional cube}\label{sec:carrier}

 We now fix the vertices of the three-dimensional cube $Q_3$, the spanning tree, and the five generators associated with the cotree edges. Let
$$
V=\{0,1\}^{3}
$$
be the vertex set of \(\Qcube\). We number the coordinates
by \(0,1,2\). A step in direction \(i\) changes the \(i\)-th coordinate
from \(0\) to \(1\), or from \(1\) to \(0\). We fix the basepoint
\(o=000\) and write a walk as
$$
w=(o;a_1a_2\cdots a_n),
\qquad a_k\in\{0,1,2\}.
$$
Let \(v_0=o,v_1,\ldots,v_n\) be the associated vertex sequence, where
\(v_k\) is obtained from \(v_{k-1}\) by changing coordinate \(a_k\).
The walk \(w\) is closed when \(v_n=o\). We write
\[
e_k=(v_{k-1},v_k)
\]
for the \(k\)-th directed edge. Thus the same walk can also be written
\(w=e_1\cdots e_n\) when its directed edges are used.
{\begin{lemma}\label{lem:closure}
A walk \(w=(o;a_1\cdots a_n)\) is closed if and only if each direction
\(0,1,2\) occurs an even number of times among
\(a_1,\ldots,a_n\).
\end{lemma}

\begin{proof}
Each step in direction \(i\) changes only the \(i\)-th coordinate. Hence
the \(i\)-th coordinate returns to its initial value if and only if
direction \(i\) is used an even number of times. Therefore \(v_n=o\) if
and only if each of the three directions occurs an even number of times.
\end{proof}}

Two walks are equal when their step sequences coincide. Concatenation of
based closed walks is  based and closed, and the empty walk
\(1=(o;\,)\) is a neutral element. Thus \(\W\) is a monoid. By
Lemma~\ref{lem:closure} it is the submonoid of the free monoid on
\(\{0,1,2\}\) consisting of the words in which each letter occurs an even
number of times.

We fix the  Gray-code spanning tree (Figure 1)
$$
000,\ 001,\ 011,\ 010,\ 110,\ 111,\ 101,\ 100,
$$
where successive vertices differ in one coordinate. The remaining five
edges are the cotree edges
$$
(0,2),\quad (0,4),\quad (1,5),\quad (3,7),\quad (4,6),
$$
where a vertex is written \(b_2b_1b_0\), with \(b_i\) the \(i\)-th
coordinate, and is read as the integer \(4b_2+2b_1+b_0\).   Thus
\(000,\ldots,111\) become \(0,\ldots,7\), and a step in direction \(i\)
changes this integer by \(\pm2^{i}\). We orient each cotree edge from the
smaller integer to the larger and denote them by
\(\gen{1},\ldots,\gen{5}\).

For a based closed walk \(w\), a forward crossing of the \(i\)-th cotree edge contributes \(x_i\), a backward crossing contributes \(x_i^{-1}\), and a tree edge contributes no letter. This gives a word \(\ell(w)\). After free
reduction we obtain the homomorphism
$$
\rho:\W\longrightarrow F_5.
$$
 Since
$$
\ell(uv)=\ell(u)\ell(v),
$$
we have
$$
\rho(uv)=\rho(u)\rho(v),\qquad \rho(1)=1.
$$
Hence \(\rho\) is a monoid homomorphism. This is the spanning tree
presentation of the fundamental group of a connected graph (see e.g.
\cite{biggs,hatcher,serre-trees}). In the sense of
Definition~\ref{def:clock}, free reduction is a sequential reading whose state is the reduced word.

{One transcription fixes the convention. The square walk
\(
(000;0101)
\)
passes through the vertices
\(
0,1,3,2,0.
\)
Its first three steps are tree edges and contribute nothing. The last step
crosses the cotree edge \((0,2)\) from the larger integer to the smaller, so it
contributes \(\gen{1}^{-1}\). Hence its reduced loop word is
\(
\gen{1}^{-1}.
\)}

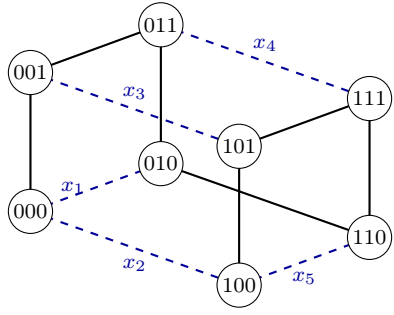
\begin{figure}[t]\label{cube}
\centering
\begin{tikzpicture}[scale=1.15, baseline]
  \tikzset{vtx/.style={circle,draw,fill=white,inner sep=1.1pt,font=\scriptsize}}
  \coordinate (v000) at (0,0);
  \coordinate (v001) at (0,1.6);
  \coordinate (v010) at (1.5,0.55);
  \coordinate (v011) at (1.5,2.15);
  \coordinate (v100) at (2.4,-0.85);
  \coordinate (v101) at (2.4,0.75);
  \coordinate (v110) at (3.9,-0.3);
  \coordinate (v111) at (3.9,1.3);
  \draw[thick] (v000)--(v001)--(v011)--(v010)--(v110)--(v111)--(v101)--(v100);
  \draw[thick,dashed,blue!60!black] (v000)--(v010) node[midway,left,font=\scriptsize]{\(\gen{1}\)};
  \draw[thick,dashed,blue!60!black] (v000)--(v100) node[midway,below,font=\scriptsize]{\(\gen{2}\)};
  \draw[thick,dashed,blue!60!black] (v001)--(v101) node[midway,above,font=\scriptsize]{\(\gen{3}\)};
  \draw[thick,dashed,blue!60!black] (v011)--(v111) node[midway,above,font=\scriptsize]{\(\gen{4}\)};
  \draw[thick,dashed,blue!60!black] (v100)--(v110) node[midway,below,font=\scriptsize]{\(\gen{5}\)};
  \foreach \p/\lab in {
    v000/000,v001/001,v010/010,v011/011,
    v100/100,v101/101,v110/110,v111/111}
    \node[vtx] at (\p) {\lab};
\end{tikzpicture}
\caption{The \(1\)-skeleton of \(Q_3\).  Solid edges form the Gray-code spanning tree. Dashed edges are the five cotree edges, labeled by
\(\gen{1},\ldots,\gen{5}\). }
\end{figure}

 {\begin{proposition}\label{prop:firstcube}
The cycle ranks of \(Q_0,Q_1,Q_2,Q_3\) are
$$
0,\ 0,\ 1,\ 5.
$$
Hence
$$
\pi_1(Q_2)\cong\Z,
\qquad
\pi_1(Q_3)\cong F_5.
$$
Thus \(Q_3\) is the first hypercube whose reduced loop words can be
noncommutative.
\end{proposition}

\begin{proof}
For a connected graph, the cycle rank is
$$
|E|-|V|+1.
$$
The graph \(Q_0\) has one vertex and no edges. For \(d\geq1\), the
\(d\)-cube has \(2^d\) vertices and \(d2^{d-1}\) edges. Hence its cycle
rank is
$$
d2^{d-1}-2^d+1,
$$
which gives the corresponding values.

The fundamental group of a connected graph is free of rank equal to its
cycle rank \cite{biggs,hatcher,serre-trees}. Thus
$$
\pi_1(Q_2)\cong F_1\cong\Z,
\qquad
\pi_1(Q_3)\cong F_5.
$$
Since \(F_1\) is abelian and \(F_5\) is nonabelian, the last statement
follows.
\end{proof}}

\subsection{\texorpdfstring{The classical \(7+5\) decomposition}
{The classical 7+5 decomposition}}

We choose an orientation of the twelve edges and let
\(C_1(Q_3;\R)\) denote the real edge space. Using the standard
Euclidean identification of \(C^1(Q_3;\R)\) with \(C_1(Q_3;\R)\),
let
$$
d:C^0(Q_3;\R)\longrightarrow C^1(Q_3;\R)
$$
be the coboundary map, so that for a vertex potential \(\phi\),
$$
(d\phi)(u,v)=\phi(v)-\phi(u)
$$
on an oriented edge \(u\to v\). Since \(Q_3\) is connected, \(\ker d\) consists of the constant
potentials. Hence
$$
\dim(\operatorname{im} d)=8-1=7.
$$
Let
$$
\partial:C_1(Q_3;\R)\longrightarrow C_0(Q_3;\R)
$$
be the boundary map. Its kernel 
is the cycle space. Hence
$$
\dim(\ker\partial)=12-8+1=5.
$$
For \(\phi\in C^0(Q_3;\R)\) and \(z\in\ker\partial\), we have
\[\langle d\phi,z\rangle=\langle\phi,\partial z\rangle=0.\] 
Thus \(\operatorname{im}d\) and \(\ker\partial\) are orthogonal, and
their dimensions add to twelve.

Accordingly, the edge space admits the orthogonal decomposition
$$
C_1(\Qcube;\R)
=
\operatorname{im} d
\oplus
\ker\partial,
\qquad
12=7+5.
$$
The map \(a=\ab\circ\rho\) given by \eqref{11} records the five cycle coordinates, i.e. the exponent sums of the five generators.  {
More precisely, let \(C_i\) be the fundamental cycle associated with the
\(i\)-th cotree edge, oriented so that this cotree edge has coefficient
\(+1\). Then \(C_1,\ldots,C_5\) form a basis of the integral cycle space.
Let \(z\) be an integral cycle and let \(b_i\) be its coefficient
on the \(i\)-th cotree edge. Then
$$
z-\sum_{i=1}^{5} b_i C_i
$$
is a cycle supported only on the spanning tree. A tree has no nonzero
cycles, so this difference is zero. Thus \(C_1,\ldots,C_5\) span the
integral cycle space. They are  independent, since \(C_i\) contains
the \(i\)-th cotree edge with coefficient \(+1\), while the other
\(C_j\) have coefficient zero on this edge.

For every based closed walk \(w\), its oriented \(1\)-chain satisfies

$$
z(w)=\sum_{i=1}^{5}a_i(w)C_i.
$$

In particular, \(a(w)=0\) if and only if \(z(w)=0\).

\section{Additive readings}\label{sec:blindness}

Let \(\A=\Z^V\) be the group of integer functions on the vertex set \(V\). For \(\phi\in\A\) and a walk \(w=(v_0,\ldots,v_n)\), define
\begin{equation}\label{eq:potential-telescope}
R_\phi(w)
=
\sum_{k=1}^{n}
\bigl(\phi(v_k)-\phi(v_{k-1})\bigr).
\end{equation}
This is the pairing of the oriented \(1\)-chain defined by the walk with the exact cochain \(d\phi\).
\begin{lemma}\label{thm:telescope}
For every \(\phi\in\A\) and every walk \(w=(v_0,\ldots,v_n)\), it holds
$$
R_\phi(w)=\phi(v_n)-\phi(v_0).
$$
In particular, \(R_\phi(w)=0\) for every closed walk.
\end{lemma}
\begin{proof}
If we expand the sum, all intermediate terms cancel, so
$$
\sum_{k=1}^{n}
\bigl(\phi(v_k)-\phi(v_{k-1})\bigr)
=
\phi(v_n)-\phi(v_0).
$$
\end{proof}
{Thus an exact potential reading depends only on the endpoints of the walk. On \(\W\), it is therefore the trivial recognizer, i.e. the zero homomorphism.

The potential reading \eqref{eq:potential-telescope} is a special additive one-step reading, with the weight of a directed edge \(p\to q\) given by
$$
g(p,q)=\phi(q)-\phi(p).
$$
We now allow arbitrary weights on directed steps. Let \(G\) be an abelian group and let
$$
g:V\times V\longrightarrow G
$$
be a weight function. For a walk \(w=(v_0,\ldots,v_n)\), define
$$
R_g(w)
=
\sum_{k=1}^{n} g(v_{k-1},v_k).
$$
We call \(R_g\) {\it an additive one-step reading}. Thus \(R_g(w)\) is the
accumulated weight along the walk. We call \(g\) {\it antisymmetric} if
$$
g(q,p)=-g(p,q)
$$
for every edge \(\{p,q\}\).   Then opposite traversals cancel, so \(R_g\) depends only on the oriented
\(1\)-chain:
\(
R_g(w)=\langle g,z(w)\rangle,
\)
where \(g\) is a \(1\)-cochain.
 
{For \(p,q\in V\), define the directed transition count by}

$$
\tau_w(p,q)
=
\#\{k:(v_{k-1},v_k)=(p,q)\}.
$$
}

{\begin{proposition}\label{thm:additive}
Every additive one-step reading factors through the directed transition
counts:
$$
R_g(w)
=
\sum_{(p,q)\in V\times V}
\tau_w(p,q)\,g(p,q).
$$
Consequently, if 
$
\tau_u=\tau_v$ then $R_g(u)=R_g(v)$
for every abelian group \(G\) and every weight function \(g\).
\end{proposition}

\begin{proof}
Each directed step \((p,q)\) occurs exactly \(\tau_w(p,q)\) times in the walk.
By collecting equal terms, we obtain
$$
R_g(w)
=
\sum_{(p,q)\in V\times V}
\tau_w(p,q)\,g(p,q).
$$
Hence \(R_g(w)\) depends on the directed transition counts and not on
their order.
\end{proof}
}

Proposition~\ref{thm:additive} applies to one-step readings with values in an
abelian group. It does not apply to arbitrary functions of the ordered
sequence, nor to noncommutative updates.
\medskip

Let
\begin{equation}\label{eq:AB-blocks}
A=(000;0101),
\qquad
B=(000;1212).
\end{equation}
be two  square loops, and define
\[
w_{AB}=AB,
\qquad
w_{BA}=BA.
\]
Both walks are closed and have length eight.

\begin{example}\label{ex:ABBA}
The walk \(A\) is the square loop in coordinates \(0\) and \(1\), and
\(B\) is the square loop in coordinates \(1\) and \(2\), both based at
\(000\). Hence \(AB\) and \(BA\) traverse the same two loops, but in
different order (Figure 2). {The corresponding vertex sequences are
$$
0,1,3,2,0,2,6,4,0
\qquad\text{and}\qquad
0,2,6,4,0,1,3,2,0 .
$$}
\end{example}

\begin{theorem}\label{thm:order}
The walks \(w_{AB}\) and \(w_{BA}\) have the same directed transition
counts but different reduced loop words:
\[
\tau_{w_{AB}}=\tau_{w_{BA}},
\qquad
\rho(w_{AB})\neq\rho(w_{BA}).
\]
\end{theorem}

\begin{proof}
Transition counts are additive under concatenation, so we have
$$
\tau_{AB}=\tau_A+\tau_B=\tau_B+\tau_A=\tau_{BA}.
$$
We have
\[
\rho(A)=\gen{1}^{-1},
\qquad
\rho(B)=\gen{1}\gen{5}^{-1}\gen{2}^{-1}.
\]
Therefore, we obtain
$$
\rho(AB)=\gen{5}^{-1}\gen{2}^{-1},
\qquad
\rho(BA)=
\gen{1}\gen{5}^{-1}\gen{2}^{-1}\gen{1}^{-1}.
$$
These are distinct reduced words in \(F_5\).
\end{proof}
\begin{figure}[t]\label{cube2}
\centering
\begin{tikzpicture}[scale=0.92,
  vtx/.style={circle,draw,fill=white,inner sep=0.9pt,font=\tiny},
  sA/.style={-stealth,thick,blue!60!black},
  sB/.style={-stealth,thick,red!60!black},
  num/.style={font=\tiny,inner sep=1pt}]

\begin{scope}
  \coordinate (a000) at (0,0);      \coordinate (a001) at (0,1.6);
  \coordinate (a010) at (1.5,0.55); \coordinate (a011) at (1.5,2.15);
  \coordinate (a100) at (2.4,-0.85);\coordinate (a110) at (3.9,-0.3);
  \draw[sA] (a000)--(a001) node[num,midway,left]{1};
  \draw[sA] (a001)--(a011) node[num,midway,above]{2};
  \draw[sA] (a011)--(a010) node[num,midway,right]{3};
  \draw[sA,dashed] (a010) to[bend left=16] node[num,midway,above left]{4} (a000);
  \draw[sB,dashed] (a000) to[bend left=16] node[num,midway,below right]{5} (a010);
  \draw[sB] (a010)--(a110) node[num,midway,above]{6};
  \draw[sB,dashed] (a110)--(a100) node[num,midway,below]{7};
  \draw[sB,dashed] (a100)--(a000) node[num,midway,below left]{8};
  \foreach \p/\l in {a000/000,a001/001,a010/010,a011/011,a100/100,a110/110}
    \node[vtx] at (\p) {\l};
  \node[font=\small] at (1.7,-1.6) {\(w_{AB}=AB\)};
\end{scope}

\begin{scope}[shift={(5.6,0)}]
  \coordinate (b000) at (0,0);      \coordinate (b001) at (0,1.6);
  \coordinate (b010) at (1.5,0.55); \coordinate (b011) at (1.5,2.15);
  \coordinate (b100) at (2.4,-0.85);\coordinate (b110) at (3.9,-0.3);
  \draw[sB,dashed] (b000) to[bend left=16] node[num,midway,below right]{1} (b010);
  \draw[sB] (b010)--(b110) node[num,midway,above]{2};
  \draw[sB,dashed] (b110)--(b100) node[num,midway,below]{3};
  \draw[sB,dashed] (b100)--(b000) node[num,midway,below left]{4};
  \draw[sA] (b000)--(b001) node[num,midway,left]{5};
  \draw[sA] (b001)--(b011) node[num,midway,above]{6};
  \draw[sA] (b011)--(b010) node[num,midway,right]{7};
  \draw[sA,dashed] (b010) to[bend left=16] node[num,midway,above left]{8} (b000);
  \foreach \p/\l in {b000/000,b001/001,b010/010,b011/011,b100/100,b110/110}
    \node[vtx] at (\p) {\l};
  \node[font=\small] at (1.7,-1.6) {\(w_{BA}=BA\)};
\end{scope}
\end{tikzpicture}
\caption{The square faces \(A\) and \(B\) concatenated in both orders,
with the step order marked and cotree edges dashed. The two walks
traverse the same directed edges, so \(\tau_{w_{AB}}=\tau_{w_{BA}}\),
while their reduced loop words differ.}
\label{fig:abba}
\end{figure}
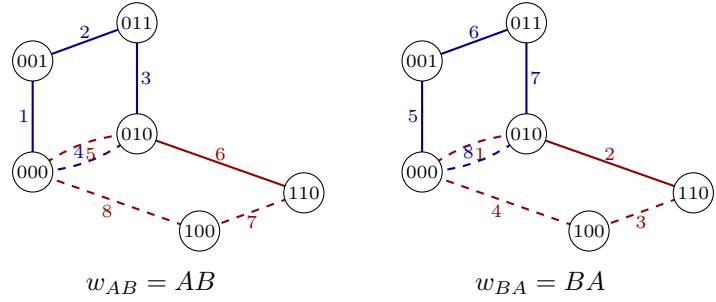

Thus no reading determined by the directed transition counts can recover
the reduced loop word. Since a walk \(w=e_1\cdots e_n\) records the order
of its directed edges, an arbitrary function of \(w\) may depend on this
order. The claim is only that commutative aggregation does not.
\section{The fourteen-step witness}\label{sec:fourteen}

The pair \(w_{AB},w_{BA}\) shows that additive readings do not detect order, but each walk has
nonzero abelian loop content. We now give a stronger witness. Its five
signed cycle counts are all zero, but it has nontrivial
noncommutative information.

Let \(w\) be a walk of length \(n\). Write
$$
 a_i(w)\in\Z,\qquad 1\leq i\leq5,
$$
for the signed traversal count of the \(i\)-th cotree edge. Let
\(s_i^{(k)}\in\{-1,0,1\}\) be the contribution of step \(k\) to this count, and set
$$
a_i^{(0)}=0,
\qquad
a_i^{(k)}=\sum_{\ell=1}^{k}s_i^{(\ell)}.
$$
Thus \(a_i^{(k-1)}\) is the signed count before step \(k\),  and
\(
a_i^{(n)}=a_i(w)
\). We define the
antisymmetric degree two coordinate
$$
 \mu_{ij}(w)
 =
 \sum_{k=1}^{n}
 \left(
 a_i^{(k-1)}s_j^{(k)}
 -
 a_j^{(k-1)}s_i^{(k)}
 \right).
$$
We write
$$
\mu(w)=\bigl(\mu_{ij}(w)\bigr)_{1\leq i<j\leq5}\in\Z^{10}.
$$
{Thus \(\mu_{ij}\) records the order interaction between the \(i\)-th and
\(j\)-th cotree crossings: each \(j\)-crossing is weighted by the current
\(i\)-count, and each \(i\)-crossing by the current \(j\)-count with the
opposite sign.

We use the usual commutator convention
\(
[x,y]=xyx^{-1}y^{-1}.
\)

{\begin{lemma}\label{lem:magnus}
The coordinates \(\mu_{ij}\) are invariant under free reduction, so
\(\mu(w)\) depends only on the reduced loop word \(\rho(w)\). If \(a(w)=0\), then each
\(\mu_{ij}(w)\) is even and, under the isomorphism
\[
\gamma_2(F_5)/\gamma_3(F_5)\cong\Lambda^2\Z^5,
\]
where \([x_i,x_j]\) corresponds to \(x_i\wedge x_j\), the class of
\(\rho(w)\) has coefficient \(\mu_{ij}(w)/2\) on
\(x_i\wedge x_j\) for \(i<j\).
In particular, if \(a(w)=0\), then
\[
\rho(w)\in\gamma_3(F_5)
\quad\Longleftrightarrow\quad
\mu(w)=0.
\]
\end{lemma}
\begin{proof}
For invariance, insert an adjacent pair \(x_mx_m^{-1}\) into the cotree
letter sequence. Fix \(i<j\). If \(m\notin\{i,j\}\), the inserted letters
give no contribution to \(\mu_{ij}\). If \(m=i\), the two contributions are
\[
-a_j^{(k-1)}
\qquad\text{and}\qquad
+a_j^{(k-1)},
\]
while if \(m=j\), they are
\[
+a_i^{(k-1)}
\qquad\text{and}\qquad
-a_i^{(k-1)}.
\]
In both cases they cancel. After the inserted pair, all running counts have
their previous values, so no later term is affected. The case
\(x_m^{-1}x_m\) is identical. Therefore \(\mu\) is invariant under free
reduction, and hence depends only on \(\rho(w)\).

For the parity, we write
\[
e_{ij}
=
\sum_{k=1}^{n} a_i^{(k-1)}s_j^{(k)}.
\]
At step \(k\), we have
\[
a_i^{(k)}
=
a_i^{(k-1)}+s_i^{(k)},
\qquad
a_j^{(k)}
=
a_j^{(k-1)}+s_j^{(k)}.
\]
Hence
\[
a_i^{(k)}a_j^{(k)}
-
a_i^{(k-1)}a_j^{(k-1)}
=
a_i^{(k-1)}s_j^{(k)}
+
a_j^{(k-1)}s_i^{(k)}
+
s_i^{(k)}s_j^{(k)}.
\]
One step crosses at most one cotree edge, so
\[
s_i^{(k)}s_j^{(k)}=0
\qquad (i\neq j).
\]
Summing over \(k\), and using
\(a_i^{(0)}=a_j^{(0)}=0\), we obtain
\[
e_{ij}+e_{ji}=a_i(w)a_j(w).
\]
From the definition, we also have
\[
\mu_{ij}=e_{ij}-e_{ji}.
\]
If \(a(w)=0\), then \(e_{ji}=-e_{ij}\), and therefore
\[
\mu_{ij}=2e_{ij}.
\]
In particular, each \(\mu_{ij}(w)\) is even.

If \(a(w)=0\), then \(\rho(w)\in\gamma_2(F_5)\). By the degree-two
case of the Magnus expansion
(see e.g. \cite{magnus1935,chen-fox-lyndon,magnus-karrass-solitar,gadish}),
the coefficient of \(x_i\wedge x_j\) in the class of \(\rho(w)\) is
\(e_{ij}=\mu_{ij}(w)/2\). Hence this class is zero if and only if
\(\mu(w)=0\). Therefore,
\[
\rho(w)\in\gamma_3(F_5)
\quad\Longleftrightarrow\quad
\mu(w)=0.
\]
\end{proof}}

We measure length in the \(Q_3\) edge metric, before spanning-tree
reduction. This differs from shortest word results in the free
generator metric
\cite{malestein-putman,elkasapy-thom,malestein-putman-2024}.

When \(a(w)=0\), geometrically \(\mu_{ij}/2\) is a signed area coordinate. For smooth
paths the same term appears in the path signature
\cite{hambly-lyons,chevyrev-lyons}. For a discrete sequence of
increments it is the degree-two coordinate of the iterated sums
signature \cite{diehl-ebrahimi-fard-tapia}.
\medskip 
{Let us consider the walk
\begin{equation}\label{walk}
W=
(000;\,
01012020101202).
\end{equation}
Its vertex sequence is
\[
0,1,3,2,0,4,5,1,0,2,3,1,5,4,0.
\]

\begin{theorem}\label{thm:witness}
The walk \(W\) defined by \eqref{walk} is closed and satisfies
\(
a(W)=0,
\)
while
\(
\mu_{12}(W)=-2.
\)
Its reduced loop word is
\[
\rho(W)
=
\gen{1}^{-1}\gen{2}\gen{3}^{-1}\gen{1}\gen{3}\gen{2}^{-1}.
\]
In particular,
\(
\rho(W)\neq1
\)
in \(F_5\).
\end{theorem}
\begin{proof}
If we run the displayed fourteen steps (Figure 3), the walk returns to \(0\). Its
cotree letter sequence is
\[
\gen{1}^{-1},\ \gen{2},\ \gen{3}^{-1},\
\gen{1},\ \gen{3},\ \gen{2}^{-1}.
\]
We have that each exponent sum is zero, and the displayed word has no
adjacent inverse pair. Hence it is freely reduced, and therefore
\[
\rho(W)=\gen{1}^{-1}\gen{2}\gen{3}^{-1}\gen{1}\gen{3}\gen{2}^{-1}.
\]
For the pair \((1,2)\), the six cotree letters contribute, in order,
\[
0,\ -1,\ 0,\ -1,\ 0,\ 0,
\]
so we obtain \(\mu_{12}(W)=-2\). 
\end{proof}}
 \begin{figure}[t]
\centering
\begin{tikzpicture}[xscale=0.86,yscale=0.8,
  bx/.style={draw,minimum width=.60cm,minimum height=.52cm,font=\small},
  sm/.style={font=\tiny}]
\foreach \k/\ax/\vv/\ltr/\mm in {%
  1/0/1/{}/{},
  2/1/3/{}/{},
  3/0/2/{}/{},
  4/1/0/{\(\gen{1}^{-1}\)}/{\(0\)},
  5/2/4/{\(\gen{2}\)}/{\(-1\)},
  6/0/5/{}/{},
  7/2/1/{\(\gen{3}^{-1}\)}/{\(0\)},
  8/0/0/{}/{},
  9/1/2/{\(\gen{1}\)}/{\(-1\)},
  10/0/3/{}/{},
  11/1/1/{}/{},
  12/2/5/{\(\gen{3}\)}/{\(0\)},
  13/0/4/{}/{},
  14/2/0/{\(\gen{2}^{-1}\)}/{\(0\)}}
{
  \node[sm] at (\k,1.0)   {\k};
  \node[bx] at (\k,0.4)   {\ax};
  \node[sm] at (\k,-0.2)  {\vv};
  \node[sm] at (\k,-0.85) {\ltr};
  \node[sm] at (\k,-1.45) {\mm};
}
\node[sm,anchor=east] at (0.4,1.0)   {step};
\node[sm,anchor=east] at (0.4,0.4)   {direction};
\node[sm,anchor=east] at (0.4,-0.2)  {vertex};
\node[sm,anchor=east] at (0.4,-0.85) {cotree letter};
\node[sm,anchor=east] at (0.4,-1.45) {term in \(\mu_{12}\)};
\end{tikzpicture}
\caption{The fourteen-step walk \(W\) of Theorem~\ref{thm:witness}: for
each step, the direction used, the vertex reached in decimal code, the
cotree letter contributed, and its term in \(\mu_{12}\). The six letters
spell \(\gen{1}^{-1}\gen{2}\gen{3}^{-1}\gen{1}\gen{3}\gen{2}^{-1}\), all
exponent sums vanish, so \(a(W)=0\) while \(\rho(W)\neq1\) and
\(\mu_{12}(W)=-2\).}
\label{fig:witness}
\end{figure}

Since \(W\) is closed and \(a(W)=0\), Section~\ref{sec:carrier} gives
\(z(W)=0\). Hence
$
R_g(W)=0
$
for every antisymmetric weight \(g\). Thus every antisymmetric additive reading, in particular every potential reading, takes the value \(0\) on \(W\), as on the empty walk. Additive readings that are not antisymmetric can separate \(W\) from the empty walk: for \(G=\Z\), the constant weight
\(g\equiv1\) gives \(R_g(W)=14\).

{\begin{lemma}\label{lem:nok23}
For every \(n\geq1\), the graph \(Q_n\) contains no subgraph isomorphic to
\(K_{2,3}\).
\end{lemma}

\begin{proof}
Let \(p,q\) be distinct vertices of \(Q_n\), and let \(r\) be a common
neighbour of \(p\) and \(q\). Then \(r\) differs from each of them in one
coordinate, so \(p\) and \(q\) differ in exactly two coordinates, let us say
\(i\) and \(j\). In that case \(r\) is obtained from \(p\) by changing
coordinate \(i\), or by changing coordinate \(j\). Hence any two distinct
vertices of \(Q_n\) have at most two common neighbours.

In \(K_{2,3}\) the two vertices of the part of size two have three common
neighbours. Therefore \(K_{2,3}\) is not a subgraph of \(Q_n\).
\end{proof}}

\begin{theorem}\label{thm:fourteen-minimal}
Among based closed walks on \(\Qcube\), the shortest walk satisfying
\[
a(w)=0
\quad\text{and}\quad
\mu(w)\neq0
\]
has length \(14\). Equivalently,
\[\min\bigl\{|w|:\rho(w)\in\gamma_2(F_5)\setminus\gamma_3(F_5)\bigr\}=14.\]
The walk \(W\) realizes this minimum.
\end{theorem}

\begin{proof}
The walk \(W\) of Theorem~\ref{thm:witness} has length \(14\) and
satisfies
\[
a(W)=0,
\qquad
\mu(W)\neq0.
\]
Hence the minimum is at most \(14\).

We prove the lower bound. Let \(w\) be a based closed walk such that
\[
a(w)=0,
\qquad
\mu(w)\neq0.
\]
Let \(H\) be the subgraph of \(Q_3\) formed by all edges which occur in
\(w\). Since \(w\) is a walk, \(H\) is connected. Also, since
\(a(w)=0\),  we have
\(
z(w)=0.
\)

We first show that the cycle rank of \(H\) is at least two. Suppose
first that the cycle rank is zero. Then \(H\) is a tree. Every closed
walk in a tree is null-homotopic, so we have
\(
\rho(w)=1.
\)
Since \(\mu(w)\) depends only on \(\rho(w)\), it follows that
\(\mu(w)=0\), which is a contradiction.

Suppose now that the cycle rank of \(H\) is one. Then
\[
\pi_1(H)\cong\Z,
\qquad
H_1(H;\Z)\cong\Z,
\]
and the abelianization map
\[
\pi_1(H)\longrightarrow H_1(H;\Z)
\]
is an isomorphism. The chain \(z(w)\) is supported on \(H\), and we
have \(z(w)=0\). Hence the class of \(w\) is zero in
\(H_1(H;\Z)\). We obtain that \(w\) represents the trivial element of
\(\pi_1(H)\), and therefore also the trivial element of
\(\pi_1(Q_3,o)\). Thus
\(
\rho(w)=1,
\)
which again gives \(\mu(w)=0\), a contradiction. Therefore the cycle
rank of \(H\) is at least two.

Put
\[
v=|V(H)|,
\qquad
e=|E(H)|.
\]
Since \(H\) is connected, its cycle rank is
\(
e-v+1.
\)
Hence
\[
e-v+1\geq2,
\]
and we obtain
\(
e\geq v+1.
\)

We show that \(e\geq7\). Suppose that \(e\leq6\). Then \(v\leq5\).
The graph \(H\) is bipartite, since it is a subgraph of \(Q_3\). If
\(v\leq4\), then
\[
e\leq
\left\lfloor\frac{v^2}{4}\right\rfloor
<
v+1,
\]
which contradicts \(e\geq v+1\). Hence, it follows that
\[
v=5,
\qquad
e=6.
\]
A bipartite graph on five vertices with six edges must have bipartition
sizes \(2\) and \(3\), and all six edges between the two parts are
present. Hence
it is \(K_{2,3}\). This is excluded by Lemma~\ref{lem:nok23}. Hence
\[
e\geq7.
\]  

Finally, \(z(w)=0\) means that every edge of \(H\) is traversed the
same number of times in both directions. Since every edge of \(H\)
occurs in \(w\), every edge is traversed at least once in each
direction. Therefore
\[
|w|\geq2e\geq14.
\]
Together with the walk \(W\), we obtain that the minimum is \(14\).

For the equivalent formulation, we have
\[
a(w)=0
\quad\Longleftrightarrow\quad
\rho(w)\in\gamma_2(F_5).
\]
Under this condition, Lemma~\ref{lem:magnus} gives
\[
\mu(w)\neq0
\quad\Longleftrightarrow\quad
\rho(w)\notin\gamma_3(F_5).
\]
Finally, we obtain
\[
\min\bigl\{
|w|:\rho(w)\in\gamma_2(F_5)\setminus\gamma_3(F_5)
\bigr\}=14.
\]
\end{proof}

\begin{remark}\label{rem:tree-free}
The value \(14\) does not depend on the chosen spanning tree. For two
spanning trees \(T,T'\), the corresponding free bases of
\(\pi_1(Q_3,o)\) differ by an automorphism \(\alpha\) of \(F_5\). Hence
\[
\rho_{T'}=\alpha\circ\rho_T.
\]
Since the terms of the lower central series are characteristic, we have
\[
\rho_T(w)\in\gamma_2(F_5)\setminus\gamma_3(F_5)
\quad\Longleftrightarrow\quad
\rho_{T'}(w)\in\gamma_2(F_5)\setminus\gamma_3(F_5).
\]

Also, the condition \(a(w)=0\) is tree-independent, since it is
equivalent to \(z(w)=0\). Under this condition,
Lemma~\ref{lem:magnus} gives that \(\mu(w)=0\) is also
tree-independent. The individual coordinates \(\mu_{ij}(w)\) may
depend on the tree. In particular,
\[
\mu_{12}(W)=-2
\]
refers to the tree fixed in Section~\ref{sec:carrier}.
\end{remark}


\section{Memory and order resolution}\label{sec:hierarchy}

The previous sections show what is lost by commutative summaries. We now
consider finite sequential memory and an unbounded stack.

 \subsection{Two-state separation from additive readings}

 A \emph{\(k\)-state reading} is a sequential reading in the sense of
Definition~\ref{def:clock} whose state set has at most \(k\) elements. 

{\begin{theorem}\label{thm:twostate}
There is a two-state reading \(D\) such that
\[
D(w_{AB})\neq D(w_{BA}).
\]
\end{theorem}

\begin{proof}
The directed edges of \(A\) are
\[
0\to1,\quad 1\to3,\quad 3\to2,\quad 2\to0,
\]
and the directed edges of \(B\) are
\[
0\to2,\quad 2\to6,\quad 6\to4,\quad 4\to0.
\]
We set
\(
e_A=(0,1)\) and
\(e_B=(0,2).
\)
Then \(e_A\) occurs in \(A\) only,
and \(e_B\) in \(B\) only, each exactly once.

We take
\[
S=\{0,1\},
\qquad
s_0=0,
\qquad
q(s)=s,
\]
and define
\[
T_{e_A}(s)=0,
\qquad
T_{e_B}(s)=1,
\qquad
T_e(s)=s
\]
for every other directed edge \(e\).

Thus the final state is determined by which of \(e_A\) and \(e_B\) occurs
last. Therefore, we obtain
\[
D(w_{AB})=1,
\qquad
D(w_{BA})=0.
\]
\end{proof}}

We next show that no fixed finite-state bound recovers the whole reduced
loop word. For \(B\) as in \eqref{eq:AB-blocks}, let \(B^m\) denote its \(m\)-fold
concatenation, a closed walk of length \(4m\).
\begin{theorem}\label{thm:window}
Let \(k\geq1\).  For every \(k\)-state reading \(F\),
$$
F\bigl(B^{k!}\bigr)=F\bigl(B^{2\cdot k!}\bigr),
$$
while
$$
\rho\bigl(B^{k!}\bigr)\neq\rho\bigl(B^{2\cdot k!}\bigr).
$$
\end{theorem}

{\begin{proof}
Let \(S\) be the state set of \(F\), with \(|S|\leq k\), and let
\(
t:S\longrightarrow S
\)
be the state update obtained by reading one copy of \(B\). Starting from
the initial state \(s_0\), consider
\[
s_0,\ t(s_0),\ldots,t^k(s_0).
\]
Since \(|S|\leq k\), there exist \(0\leq i<j\leq k\) such that
\[
t^i(s_0)=t^j(s_0).
\]
Set \(p=j-i\). Then
\[
t^{n+p}(s_0)=t^n(s_0)
\qquad\text{for every }n\geq i.
\]
Since \(1\leq p\leq k\), we have \(p\mid k!\), and since \(i\leq k\leq k!\),
it follows that
\[
t^{k!}(s_0)=t^{2\cdot k!}(s_0).
\]
Applying the output map gives
\[
F\bigl(B^{k!}\bigr)=F\bigl(B^{2\cdot k!}\bigr).
\]

On the other hand, we have
\[
\rho(B)=\gen{1}\gen{5}^{-1}\gen{2}^{-1},
\]
which is a nontrivial cyclically reduced word. Hence
\[
\rho(B^m)=\rho(B)^m
\]
is reduced for every \(m\geq1\). Therefore
\[
\rho\bigl(B^{k!}\bigr)
\neq
\rho\bigl(B^{2\cdot k!}\bigr).
\]
\end{proof}}
This factorial witness is not purely noncommutative, since
\[
a(B)=(1,-1,0,0,-1)\neq0,
\qquad
a(B^m)=m\,a(B),
\]
so abelianization already separates the two powers.

\subsection{The stack recovers the reduced word}

 \begin{theorem}\label{thm:stack}
There is a sequential reading with an unbounded stack, whose state set is
the set of finite reduced words, which recovers
\(\rho(w)\) for every \(w\in\W\).
\end{theorem}

\begin{proof}
We take the state set to be the set of finite reduced words in
\[
\gen{1}^{\pm1},\ldots,\gen{5}^{\pm1},
\]
with the empty word as the initial state and the last letter as the top of
the stack. The output map is the identity.

A tree edge leaves the state unchanged. If a cotree edge contributes a
letter \(l\), we delete the last letter when it is \(l^{-1}\), and
otherwise we append \(l\).

By induction on the initial segments of \(w\), after each step the stack
contains the freely reduced cotree word of the segment already read.
Therefore, after the whole walk is read, the stack contains
\(\rho(w)\).
\end{proof}

\begin{corollary}\label{cor:hierarchy}
There is a pair of walks which is not distinguished by any additive
reading, but is distinguished by a two-state reading. For every \(k\geq1\),
the walks
\[
B^{k!},\qquad B^{2\cdot k!}
\]
are not distinguished by any \(k\)-state reading, while the stack
distinguishes them. Hence no fixed finite-state bound recovers
\(\rho\) on all walks, whereas the stack recovers \(\rho(w)\) for every
\(w\in\W\).
\end{corollary}

\begin{proof}
The first statement follows from
Theorems~\ref{thm:order}, \ref{thm:twostate},  and Proposition~\ref{thm:additive}, and the second from
Theorems~\ref{thm:window} and~\ref{thm:stack}.
\end{proof}

\section{A proper order-sensitive congruence}\label{sec:congruence}

We compare the kernel congruences induced by the maps \(\rho\) and \(a=\ab\circ\rho\) of \eqref{11}, i.e.
\begin{equation}\label{ra}
\rho:\W\longrightarrow F_5
\qquad\text{and}\qquad
a:\W\longrightarrow\Z^5.
\end{equation}
Recall that \(a(w)=(a_1(w),\ldots,a_5(w))\) records the signed cotree
traversal counts.

\begin{theorem}\label{thm:proper}
We have
\[
\Delta_{\W}
\subsetneq
\sim_{\rho}
\subsetneq
\sim_{a}.
\]
\end{theorem}
 
{ \begin{proof}
Since \(a=\ab\circ\rho\), we have
\[
\rho(u)=\rho(v)\quad\Longrightarrow\quad a(u)=a(v),
\]
and therefore
\[
\sim_{\rho}\subseteq\sim_a.
\]
The inclusion is strict by Theorem~\ref{thm:order}, since
\[
\rho(w_{AB})\neq\rho(w_{BA}),
\]
while
\[
a(w_{AB})
=a(A)+a(B)
=a(B)+a(A)
=a(w_{BA}).
\]

For the first inclusion, let
\[
\beta=(000;00)
\]
be the two-step backtrack along a tree edge. Since both steps are tree
edges, we have
\[
\rho(\beta)=1=\rho(1),
\]
while \(\beta\neq1\) in \(\W\), since their step sequences are different.
Hence
\[
\Delta_{\W}\subsetneq\sim_{\rho}.
\]
\end{proof}}
Thus \(\sim_\rho\) is a proper congruence on \(\W\). By Theorem~\ref{thm:order},
it is order-sensitive: it separates \(w_{AB}\) from \(w_{BA}\), although
these walks have the same directed transition counts.

{\begin{corollary}\label{cor:quotients}
The recognition quotients determined by \(\rho\) and \(a\) given by \eqref{ra} satisfy
\[
\W/{\sim_\rho}\cong F_5,
\qquad
\W/{\sim_a}\cong\Z^5.
\]
\end{corollary}

\begin{proof}
By the first isomorphism theorem for monoids,
\[
\W/{\sim_\rho}\cong \operatorname{im}\rho,
\qquad
\W/{\sim_a}\cong \operatorname{im}a.
\]
For each cotree edge, the corresponding based fundamental loop maps to
\(\gen{i}\), while its reverse maps to \(\gen{i}^{-1}\). Hence
\(\rho:\W\to F_5\) is surjective. Thus,
\[
\W/{\sim_\rho}\cong F_5.
\]
Since \(a=\ab\circ\rho\) and both \(\rho\) and \(\ab\) are surjective,
\(a:\W\to\Z^5\) is also surjective. Therefore
\[
\W/{\sim_a}\cong\Z^5.
\]
\end{proof}}


\begin{proposition}\label{thm:commblind}
Let \(M\) be a commutative monoid and let \(f:\W\to M\) be a homomorphism.
Then
\[
f(uv)=f(vu)
\]
for all \(u,v\in\W\). In particular,
\[
f(w_{AB})=f(w_{BA}).
\]
\end{proposition}

\begin{proof}
Since \(M\) is commutative, we have
\[
f(uv)=f(u)f(v)=f(v)f(u)=f(vu).
\]
Setting \(u=A\) and \(v=B\) we obtain $f(w_{AB})=f(w_{BA}).$
\end{proof}

Thus a monoid recognizer separating \(w_{AB}\) from \(w_{BA}\) cannot
take values in a commutative monoid. By Theorem~\ref{thm:order},
\(\rho\) is such a recognizer.

\begin{remark}\label{rem:tau}
The maps \(\rho\) and \(\tau\) are incomparable. By
Theorem~\ref{thm:order}, we have
\[
\tau_{w_{AB}}=\tau_{w_{BA}},
\qquad
\rho(w_{AB})\neq\rho(w_{BA}),
\]
so \(\tau\) does not determine \(\rho\). Conversely, for
\(\beta=(000;00)\), we have
\[
\rho(\beta)=\rho(1),
\qquad
\tau_\beta\neq\tau_1,
\]
so \(\rho\) does not determine \(\tau\).
\end{remark}

A second proper order-sensitive congruence is constructed in
Section~\ref{sec:settlement}.

\section{A declared quarter-turn clock}
\label{sec:settlement}

The reduced loop word \(\rho\) gives one order-sensitive recognizer.
We now define a second one by a declared noncommutative transport on the
three coordinate directions of the cube.

Let
\[
\mathbb H_{\Z}
=
\{a+b\mathbf i+c\mathbf j+d\mathbf k:
a,b,c,d\in\Z\}
\]
be the multiplicative monoid of integer quaternions. The quaternion units
satisfy
\[
\mathbf i^2=\mathbf j^2=\mathbf k^2=-1,
\qquad
\mathbf i\mathbf j=\mathbf k,\quad
\mathbf j\mathbf k=\mathbf i,\quad
\mathbf k\mathbf i=\mathbf j.
\]
We assign to the three coordinate directions
\[
q_0=1+\mathbf i,\qquad
q_1=1+\mathbf j,\qquad
q_2=1+\mathbf k.
\]

The normalized elements
\[
\frac{q_0}{\sqrt2},\qquad
\frac{q_1}{\sqrt2},\qquad
\frac{q_2}{\sqrt2}
\]
are unit quaternions representing quarter-turns about the three
coordinate axes. 

For a walk
\[
w=(o;a_1\cdots a_n),
\]
define
\[
h:\W\longrightarrow\mathbb H_{\Z},
\qquad
h(w)=q_{a_1}\cdots q_{a_n}.
\] The map \(h\) uses the integer elements \(q_a\), not  unit
quaternions.
 
Let
\[
\mathcal T=\{\tau_w:w\in\W\}\subseteq\N^{V\times V},
\]
with pointwise addition. Since \(\tau_{uv}=\tau_u+\tau_v\) and
\(\tau_1=0\), the set \(\mathcal T\) is a submonoid of
\(\N^{V\times V}\).

\begin{theorem}\label{thm:quarter-clock}
The map
\[
c:\W\longrightarrow\mathcal T\times\mathbb H_{\Z},
\qquad
c(w)=(\tau_w,h(w)),
\]
is a monoid homomorphism and
\[
\Delta_{\W}\subsetneq\sim_c\subsetneq\sim_\tau.
\]
\end{theorem}

\begin{proof}
For \(u,v\in\W\), we have
\[
h(uv)=h(u)h(v),
\qquad
\tau_{uv}=\tau_u+\tau_v.
\]
Hence \(c\) is a monoid homomorphism. Since the first component of \(c\)
is \(\tau\), we have
\[
\sim_c\subseteq\sim_\tau.
\]

To show that this inclusion is strict, let us use
\[
u=(000;001122),
\qquad
v=(000;002211).
\]
The two walks have the same directed transition counts, so
\(
\tau_u=\tau_v.
\)
We have
\[
\begin{aligned}
h(u)
&=(1+\mathbf i)^2(1+\mathbf j)^2(1+\mathbf k)^2
 =8\mathbf i\mathbf j\mathbf k=-8,\\
h(v)
&=(1+\mathbf i)^2(1+\mathbf k)^2(1+\mathbf j)^2
 =8\mathbf i\mathbf k\mathbf j=8.
\end{aligned}
\]
Thus \(c(u)\neq c(v)\), and therefore
\[
\sim_c\subsetneq\sim_\tau.
\]

For the first inclusion, take
\[
r=(000;001001),
\qquad
r'=(000;100100).
\]
These are distinct based closed walks. Their directed edge sequences are
cyclic rotations, hence
\[
\tau_r=\tau_{r'}.
\]
Also,
\[
\begin{aligned}
h(r)
&=\bigl((1+\mathbf i)^2(1+\mathbf j)\bigr)^2
 =\bigl(2\mathbf i(1+\mathbf j)\bigr)^2
 =\bigl(2\mathbf i+2\mathbf k\bigr)^2
 =-8,\\
h(r')
&=\bigl((1+\mathbf j)(1+\mathbf i)^2\bigr)^2
 =\bigl((1+\mathbf j)2\mathbf i\bigr)^2
 =\bigl(2\mathbf i-2\mathbf k\bigr)^2
 =-8.
\end{aligned}
\]
Hence
\(
c(r)=c(r'),
\)
while \(r\neq r'\). Therefore, we get
\(
\Delta_{\W}\subsetneq\sim_c.
\)
\end{proof}
 
Thus \(\sim_c\) is a proper congruence on \(\W\). By
Theorem~\ref{thm:quarter-clock}, it is order-sensitive: it separates
\(u\) from \(v\), although these walks have the same directed transition
counts.


\begin{proposition}\label{prop:clock-incomparable}
The congruences \(\sim_c\) and \(\sim_\rho\) are incomparable.
\end{proposition}

\begin{proof}
Let \(u,v\) be the walks from Theorem~\ref{thm:quarter-clock}.
Both are concatenations of three two-step backtracks from the origin.
Hence
\[
\rho(u)=\rho(v)=1,
\]
while
\(
c(u)\neq c(v).
\)
Therefore, we have
\(
\sim_\rho\not\subseteq\sim_c.
\)

For the converse, we use the walks \(A\) and \(B\) from
\eqref{eq:AB-blocks}. By Theorem~\ref{thm:order}, we have
\[
\tau_{AB}=\tau_{BA},
\qquad
\rho(AB)\neq\rho(BA).
\]
Also,
\[
(1+\mathbf i)(1+\mathbf j)
=
(1+\mathbf j)(1+\mathbf k)
=
1+\mathbf i+\mathbf j+\mathbf k.
\]
Hence
\[
h(A)
=
\bigl((1+\mathbf i)(1+\mathbf j)\bigr)^2
=
\bigl((1+\mathbf j)(1+\mathbf k)\bigr)^2
=
h(B).
\]
Therefore
\[
h(AB)=h(A)h(B)=h(B)h(A)=h(BA).
\]
Together with \(\tau_{AB}=\tau_{BA}\), we obtain
\[
c(AB)=c(BA).
\]
Since
\[
\rho(AB)\neq\rho(BA),
\]
it follows that
\(
\sim_c\not\subseteq\sim_\rho.
\)
\end{proof}


The map \(h\) is not determined by the graph \(Q_3\). It depends on the three coordinate directions and on the choice of the quaternion units \(\mathbf i,\mathbf j,\mathbf k\) for these directions. A different choice can give a different homomorphism. Thus \(h\) is an additional structure on \(\W\).

 \section{Rectangular constraints on histories}\label{sec:boundary}

In the previous sections, we studied readings of a history \(w\in\W\).
We now consider constraints on pairs
\(
(\alpha,w)\in\A\times\W.
\)
We determine when such constraints are rectangular and when they are not.

Let
\(
\U=\A\times\W
\)
be the product carrier. A constraint \(P\subseteq\U\) is called
\emph{rectangular} if there exist subsets
\(P_{\A}\subseteq\A\) and \(P_{\W}\subseteq\W\) such that
\[
(\alpha,w)\in P
\quad\Longleftrightarrow\quad
\alpha\in P_{\A}\ \text{and}\ w\in P_{\W}.
\]
Equivalently,
\[
P=P_{\A}\times P_{\W}.
\]

 We first work on the full group \(\A=\Z^V\). Write
\[
\A_0=
\bigl\{\alpha\in\A:\textstyle\sum_{p\in V}\alpha(p)=0\bigr\}
\]
for the kernel of the sum map
\[
\alpha\longmapsto\sum_{p\in V}\alpha(p).
\]

\medskip
 Since every \(w\in\W\) is closed, Lemma~\ref{thm:telescope} gives
\(R_\phi(w)=0\) for every potential reading. Hence every constraint
depending only on potential readings is constant on \(\U\), and consequently
rectangular.
\medskip

For \(w=(v_0,\ldots,v_n)\in\W\), define
\(
\occ:\W\longrightarrow\A
\)
by
\[
\occ(w)(p)
=
\#\{k:0\leq k<n,\ v_k=p\}.
\]
Thus the closing vertex \(v_n=v_0\) is not counted. Define
\(\mathcal O\subseteq\U\) by
\[
(\alpha,w)\in\mathcal O
\quad\Longleftrightarrow\quad
\alpha=\occ(w).
\]
We call \(\mathcal O\) the \emph{occupation-binding constraint}.

\begin{proposition}\label{thm:occupation}
The occupation-binding constraint \(\mathcal O\) is not rectangular on
\(\U\).
\end{proposition}

\begin{proof}
Choose \(u,v\in\W\) with
\[
\occ(u)\neq\occ(v),
\]
for example the empty walk \(u=1\) and the square
\(v=(o;0101)\). Then
\[
(\occ(u),u)\in\mathcal O,
\qquad
(\occ(v),v)\in\mathcal O,
\]
while
\[
(\occ(u),v)\notin\mathcal O.
\]

Suppose that \(\mathcal O\) is rectangular. Then
\[
\mathcal O=P_{\A}\times P_{\W}
\]
for some \(P_{\A}\subseteq\A\) and \(P_{\W}\subseteq\W\).
From \((\occ(u),u)\in\mathcal O\) we obtain
\(\occ(u)\in P_{\A}\), and from
\((\occ(v),v)\in\mathcal O\) we obtain \(v\in P_{\W}\).
Hence
\[
(\occ(u),v)\in\mathcal O,
\]
which is a contradiction.
\end{proof}

\begin{proposition}\label{prop:neutral-occupation}
For \((\alpha,w)\in\A_0\times\W\),
\[
(\alpha,w)\in\mathcal O
\quad\Longleftrightarrow\quad
\alpha=0\ \text{and}\ w=1.
\]
\end{proposition}

\begin{proof}
For every \(w\in\W\), it holds
\[
\sum_{p\in V}\occ(w)(p)=|w|.
\]
If \((\alpha,w)\in\mathcal O\) and \(\alpha\in\A_0\), then
\(\alpha=\occ(w)\), so
\[
|w|
=
\sum_{p\in V}\occ(w)(p)
=
\sum_{p\in V}\alpha(p)
=
0.
\]
Hence \(w=1\) and \(\alpha=0\). The converse is immediate.
\end{proof}

\begin{remark}
Thus
\[
\mathcal O\cap(\A_0\times\W)=\{0\}\times\{1\}.
\]
Hence \(\mathcal O\) is rectangular on \(\A_0\times\W\); in this case
the restriction consists of a single point.
\end{remark}

\subsection{A criterion for rectangularity}

A constraint \(P\subseteq\U\) is called \emph{occupation-based} if
\[
(\alpha,w)\in P
\quad\Longleftrightarrow\quad
\Psi(\alpha,\occ(w))
\]
for some predicate
\[
\Psi:\A\times\A\longrightarrow
\{\mathrm{false},\mathrm{true}\}.
\]
Let \(\mathcal S\subseteq\W\) be an arbitrary subset.

Here a subset \(P\subseteq\U\) is called rectangular on
\(\A\times\mathcal S\) if
\[
P\cap(\A\times\mathcal S)=P_{\A}\times P_{\mathcal S}
\]
for some \(P_{\A}\subseteq\A\) and \(P_{\mathcal S}\subseteq\mathcal S\).
\begin{theorem} 
\label{thm:criterion}
The following statements are equivalent:
\begin{enumerate}[label=(\roman*)]
\item There is \(c_0\in\A\) such that
$
\occ(w)=c_0
$ for every \(w\in\mathcal S\).
\item Every occupation-based constraint is rectangular on
\(\A\times\mathcal S\).
\end{enumerate}
\end{theorem}

\begin{proof}
First, we assume (i). If
\[
(\alpha,w)\in P
\quad\Longleftrightarrow\quad
\Psi(\alpha,\occ(w)),
\]
then on \(\A\times\mathcal S\), we have
\[
(\alpha,w)\in P
\quad\Longleftrightarrow\quad
\Psi(\alpha,c_0).
\]

{Hence
\[
P\cap(\A\times\mathcal S)
=
P_{\A}\times\mathcal S,
\qquad
P_{\A}
=
\{\alpha\in\A:\Psi(\alpha,c_0)\}.
\]}
\medskip

Conversely, suppose (i) is not valid. Then there exist \(u,v\in\mathcal S\)
with \(\occ(u)\neq\occ(v)\). The constraint \(\mathcal O\) is
occupation-based, and the cross-pair argument of
Proposition~\ref{thm:occupation} shows that
\(\mathcal O\cap(\A\times\mathcal S)\) is not rectangular in
\(\A\times\mathcal S\).\end{proof}

For instance, let \(\mathcal S\) be the class of based closed walks
\((v_0,\ldots,v_8)\) with \(v_0,\ldots,v_7\) pairwise distinct. The Gray
cycle shows \(\mathcal S\neq\varnothing\), and \(\occ(w)(p)=1\) for every \(w\in\mathcal S\) and every \(p\in V\), so every occupation-based constraint is rectangular
there.

{ \section{Higher-dimensional cubes}

The \(1\)-skeleton of \(Q_n\) has \(2^n\) vertices and \(n2^{n-1}\)
edges. Hence its cycle rank is
\[
r_n=n2^{n-1}-2^n+1=(n-2)2^{n-1}+1.
\]
As in the proof of Proposition~\ref{prop:firstcube}, the fundamental
group of a connected graph is free of rank equal to its cycle rank.
Hence
\[
\pi_1(Q_n)\cong F_{r_n}.
\]

For every \(n\geq3\), a coordinate copy of \(Q_3\) is contained in
\(Q_n\). We extend the Gray spanning tree of this subcube to a spanning
tree \(T_n\) of \(Q_n\). Since the Gray tree already spans the eight
vertices of the subcube, no other edge of the subcube belongs to
\(T_n\). Thus its five cotree edges remain cotree edges for \(T_n\). By  \(\rho_{T_n}\) we denote  the reduced loop word map of \(Q_n\)
determined by \(T_n\), defined as in Section~\ref{sec:carrier}.

For a closed walk in this subcube, the cotree word calculated with \(T_n\)
is obtained from its \(Q_3\) cotree word by the homomorphism
\[
\iota:F_5\longrightarrow F_{r_n}
\]
which maps the five generators to distinct members of a free basis of
\(F_{r_n}\). This homomorphism is injective, since a subset of a free
basis freely generates a free subgroup. In particular, \(w_{AB}\) and
\(w_{BA}\)  have different reduced loop words in \(Q_n\), while
their directed transition counts agree. Therefore commutative aggregation does not determine the reduced loop
word on \(Q_n\) for any \(n\geq3\).

There is also a retraction
\[
\pi:F_{r_n}\longrightarrow F_5
\]
which maps these five basis elements to the corresponding generators of
\(F_5\), and every remaining basis element to \(1\). Hence
\(
\pi\circ\iota=\operatorname{id}_{F_5}.
\)}

The minimum of Theorem~\ref{thm:fourteen-minimal} is also independent of
\(n\).

\begin{theorem}\label{thm:fourteen-n}
Let \(n\geq3\). Among based closed walks on \(Q_n\),
\[
\min\bigl\{
|w|:\rho_{T_n}(w)\in\gamma_2(F_{r_n})\setminus\gamma_3(F_{r_n})
\bigr\}=14.
\]
The walk \(W\) of \eqref{walk}, read in the coordinate copy of \(Q_3\)
fixed above, realizes this minimum.
\end{theorem}

\begin{proof}
For the upper bound, all steps of \(W\) use edges of the subcube, so
\[
\rho_{T_n}(W)=\iota(\rho(W)).
\]
By Theorem~\ref{thm:fourteen-minimal}, we have
\[
\rho(W)\in\gamma_2(F_5)\setminus\gamma_3(F_5).
\]
Since \(\iota\) is a homomorphism, we have
\(
\iota(\rho(W))\in\gamma_2(F_{r_n}).
\)
Suppose that \(\iota(\rho(W))\in\gamma_3(F_{r_n})\). Then applying \(\pi\)
gives
\(
\rho(W)\in\gamma_3(F_5),
\)
which is a contradiction. Hence
\[
\rho_{T_n}(W)\in\gamma_2(F_{r_n})\setminus\gamma_3(F_{r_n}),
\]
and the minimum is at most \(14\).

For the lower bound, let \(w\) be a based closed walk on \(Q_n\) with
\[
\rho_{T_n}(w)\in\gamma_2(F_{r_n})\setminus\gamma_3(F_{r_n}).
\]
Since \(\rho_{T_n}(w)\in\gamma_2(F_{r_n})\), all signed cotree traversal
counts of \(w\) vanish, so as in Section~\ref{sec:carrier}, applied to
\(T_n\), we have
\(
z(w)=0.
\)
Let \(H\) be the subgraph of \(Q_n\) formed by all edges which occur in
\(w\). Since \(w\) is a walk, \(H\) is connected.

The cycle rank of \(H\) is at least two. Indeed, for cycle rank zero or
one, the two arguments in the proof of
Theorem~\ref{thm:fourteen-minimal} use only \(z(w)=0\) and the fact that
\(\pi_1(H)\) is free of that rank. So, we have \(\rho_{T_n}(w)=1\), hence
\(\rho_{T_n}(w)\in\gamma_3(F_{r_n})\), which is a contradiction. Putting
\[
v=|V(H)|,
\qquad
e=|E(H)|,
\]
we  obtain \(e\geq v+1\).

The graph \(H\) is bipartite, since it is a subgraph of \(Q_n\). If
\(e\leq6\), then \(v\leq5\), and the count in the proof of
Theorem~\ref{thm:fourteen-minimal} implies \(v=5\), \(e=6\), so \(H\) is
\(K_{2,3}\). This is excluded by Lemma~\ref{lem:nok23}. Hence
\(
e\geq7.
\)

Finally, \(z(w)=0\) means that every edge of \(H\) is traversed the same
number of times in both directions, and every edge of \(H\) occurs in
\(w\). Therefore
\[
|w|\geq2e\geq14.
\]
\end{proof}

For \(n=3\) we have \(r_3=5\), and \(T_3\) is the Gray spanning tree of
Section~\ref{sec:carrier}. Thus Theorem~\ref{thm:fourteen-minimal} is
the case \(n=3\).

As in Remark~\ref{rem:tree-free}, the statement does not depend on the
choice of \(T_n\), since the terms of the lower central series are
characteristic.

\section{Conclusion}

The passage from one repeatable generator to several independent loops
introduces order information. Every additive one-step reading factors
through the directed transition counts. Therefore commutative
aggregation does not recover the reduced loop word. The reduced-word recognizer gives a proper
noncommutative recognition quotient
\[
\Delta_{\W}\subsetneq\sim_\rho\subsetneq\sim_a.
\]
Thus the free-group reading retains order information which is lost by
the abelian reading.

We also determined nontrivial degree-two commutator information in the
cube edge metric. We proved that
\[
\min\bigl\{
|w|:\rho(w)\in\gamma_2(F_5)\setminus\gamma_3(F_5)
\bigr\}=14.
\]
Thus the fourteen-step witness is minimal. Its abelianization is trivial,
while its degree-two commutator information is nonzero.

We also obtain a separation between finite and unbounded sequential
memory. A two-state reading separates one order pair. For every
\(k\geq1\), there are two walks which no \(k\)-state reading separates,
although their reduced loop words are different. No fixed finite-state
bound therefore recovers \(\rho\) on all histories, while an unbounded
stack recovers \(\rho(w)\) for every \(w\in\W\).

The declared quarter-turn quaternion clock gives another proper
order-sensitive congruence,
\[
\Delta_{\W}\subsetneq\sim_c\subsetneq\sim_\tau.
\]
The congruences \(\sim_c\) and \(\sim_\rho\) are incomparable. Thus
different noncommutative recognizers can retain different order
information. The quaternion map is an additional structure on \(\W\) and is not
determined by the graph \(Q_3\).

For constraints on \(\A\times\W\), potential readings are constant
on closed histories, while the occupation-binding constraint is not
rectangular on \(\A\times\W\). On \(\A_0\times\W\), its restriction is
\[
\{0\}\times\{1\}.
\]
More generally, for a class \(\mathcal S\subseteq\W\), every
occupation-based constraint is rectangular on
\(\A\times\mathcal S\) if and only if the occupation vector is constant
on \(\mathcal S\). Hence this condition characterizes when
occupation-based constraints are rectangular.

Finally, the obstruction to commutative recovery is not specific to
\(Q_3\). For every \(n\geq3\), a coordinate copy of \(Q_3\) in \(Q_n\)
gives closed walks with equal directed transition counts and different
reduced loop words. Thus the same obstruction to commutative recovery
occurs in every hypercube \(Q_n\), \(n\geq3\). The minimum is also independent of \(n\). On every \(Q_n\) with
\(n\geq3\), the shortest based closed walk with vanishing abelianization
and nontrivial degree-two commutator information has length \(14\)
(Theorem~\ref{thm:fourteen-n}).

\bigskip

\vspace{6pt}
\noindent {\bf Author Contributions:} {Conceptualization, J.W.; Methodology, J.W. and M.Z.; Validation, J.W. and M.Z.; Formal Analysis, M.Z. and J.W.; Investigation, J.W. and M.Z.; Resources,
J.W.; Writing---Original Draft Preparation, J.W.; Writing---Review and Editing, M.Z. and J.W.; Funding Acquisition, J.W. }

\begingroup
\raggedright

\endgroup

\end{document}